\documentclass[a4paper,oneside,english,12pt,reqno]{amsart}
\usepackage{mathptmx}

\usepackage[T1]{fontenc}
\usepackage{CJKutf8}
\usepackage{amsthm}
\usepackage{amssymb}
\usepackage{fullpage}
\makeatletter

\pdfpageheight\paperheight
\pdfpagewidth\paperwidth

\numberwithin{equation}{section}
\numberwithin{figure}{section}
\theoremstyle{plain}
\newtheorem{thm}{\protect\theoremname}[section]
\newtheorem{conjecture}[thm]{Conjecture}
  \theoremstyle{definition}
  
  \theoremstyle{remark}
  
  \newtheorem{remarks}[thm]{\protect\remarknames}
  \theoremstyle{plain}
  \newtheorem{lem}[thm]{\protect\lemmaname}
  \theoremstyle{plain}
  \newtheorem{prop}[thm]{\protect\propositionname}

\newcommand{\Q}{\mathbb{Q}}
\newcommand{\R}{\mathbb{R}}
\newcommand{\N}{\mathbb{N}}
\newcommand{\C}{\mathbb{C}}
\newcommand{\Z}{\mathbb{Z}}
\newcommand{\SL}{\operatorname{SL}}
\newcommand{\GL}{\operatorname{GL}}
\renewcommand{\H}{\mathbb{H}}
\newcommand{\Tr}{\operatorname{Tr}}
\newcommand{\Nr}{\operatorname{Nr}}
\newcommand{\gen}{\operatorname{gen}}
\newcommand{\spn}{\operatorname{spn}}
\makeatother

\usepackage{babel}
\usepackage{enumitem,comment}
  \providecommand{\definitionname}{Definition}
  \providecommand{\lemmaname}{Lemma}
  \providecommand{\propositionname}{Proposition}
  \providecommand{\remarkname}{Remark}
  \providecommand{\remarknames}{Remarks}
\providecommand{\theoremname}{Theorem}

\begin{document}
\begin{CJK}{UTF8}{}%

\title{On sign changes of cusp forms and the halting of an algorithm to construct a supersingular elliptic curve with a given endomorphism ring}
\end{CJK}
\author{King Cheong Fung}
\address{Mathematics Department, University of Hong Kong, Pokfulam, Hong Kong}
\email{mrkcfung@hku.hk}
\author{Ben Kane}
\address{Mathematics Department, University of Hong Kong, Pokfulam, Hong Kong}
\email{bkane@maths.hku.hk}
\date{\today}

\begin{abstract}

Chevyrev and Galbraith recently devised an algorithm which inputs a maximal order of the quaternion algebra ramified at one prime and infinity and constructs a supersingular elliptic curve whose endomorphism ring is precisely this maximal order.  They proved that their algorithm is correct whenever it halts, but did not show that it always terminates.  They did however prove that the algorithm halts under a reasonable assumption which they conjectured to be true.  It is the purpose of this paper to verify their conjecture and in turn prove that their algorithm always halts.  

More precisely, Chevyrev and Galbraith investigated the theta series associated with the norm maps from primitive elements of two maximal orders.  They conjectured that if one of these theta series ``dominated'' the other in the sense that the $n$th (Fourier) coefficient of one was always larger than or equal to the $n$th coefficient of the other, then the maximal orders are actually the same.  We prove that this is the case.
\end{abstract}
\thanks{  The research of the second author was supported by grant project numbers 27300314 and 17302515 of the Research Grants Council.}

\maketitle

\section{Introduction}

In this paper, we investigate the construction of certain elliptic curves defined over finite fields.  For a prime $p$, let $E$ be an elliptic curve over $\mathbb{F}_{p^{2}}$.  Deuring \cite{Deuring} showed that the endomorphism ring of $E$ is either an order in an imaginary quadratic field (the \textit{ordinary} case) or an order in the quaternion algebra $B_p$ (see Section \ref{sec:quaternion}) which is ramified at $p$ and infinity (the \textit{supersingular} case).  The supersingular case is the primary interest of this paper.  To motivate one area of study related to such curves, we momentarily consider elliptic curves over a number field, in which case the endomorphism ring is either isomorphic to $\Z$ or it is isomorphic to an order in an imaginary quadratic field (the \textit{Complex Multiplication} or \textit{CM} case).  In the second case, we say that the elliptic curve has (exact) CM by this order.  Next recall that the orders of an imaginary quadratic field are entirely determined by their discriminants; that is to say, for each discriminant $d<0$, there is a unique order $\mathcal{O}_d$ of discriminant $d$ in the ring of integers $\mathcal{O}_{\Q(\sqrt{d})}$ of $\Q(\sqrt{d})$.  When $p$ is a prime of good reduction, there is a natural reduction map from elliptic curves over the Hilbert class field of $\Q(\sqrt{d})$ (a certain number field) to elliptic curves over $\mathbb{F}_{p^2}$.  Moreover, when $p$ is inert or ramified in $\Q(\sqrt{d})$, this map sends CM elliptic curves to supersingular elliptic curves.  An interesting question arises from this connection.  Namely, for which $d$ is the reduction map from the set of elliptic curves with CM by $\mathcal{O}_d$ to supersingular elliptic curves surjective?  This question was studied by a number of authors (cf. \cite{E-O} and \cite{J-Kane}).  It turns out that the reduction map is not always surjective and is not in general one-to-one.  Different authors have also approached the question in different directions and from slightly different perspectives.  Elkies, Ono and Yang \cite{E-O} worked on the question when the discriminant $d$ was restricted to be fundamental.  In other words, they considered those elliptic curves with exact CM by the ring of integers $\mathcal{O}_{\Q(\sqrt{d})}$ of an imaginary quadratic field and varied the field.  They proved that for $d$ sufficiently large, the image of the reduction map is surjective and furthermore that it is equidistributed across all supersingular elliptic curves.  A slight modification of this was investigated by Jetchev and the second author \cite{J-Kane}, where it was shown that the reduction from curves with exact CM by $\mathcal{O}_d$ is surjective for $d$ sufficiently large but not necessarily fundamental (albeit with some minor restriction on the choice of $d$).  The approach taken in \cite{E-O} and \cite{J-Kane} was to use a  correspondence between elliptic curves with CM by $\mathcal{O}_d$ which reduce to a supersingular elliptic curve and \textit{optimal embeddings} of $\mathcal{O}_d$ in its endomorphism ring; roughly speaking, if $\mathcal{O}_{d}$ embeds into the quaternionic order, then $\mathcal{O}_{dr^2}$ also embeds by multiplying by $r$, and optimal embeddings are those which do not come from smaller discriminants.  These optimal embeddings, in turn, correspond to primitive representations of $d$ by the norm map on trace zero elements in the quaternionic order.  

Having given one area of study centered around supersingular elliptic curves, we return to the study of supersingular elliptic curves themselves.  Chevyrev and Galbraith \cite{C-G} constructed an algorithm to compute a supersingular elliptic curve with a given endomorphism ring (a maximal order in the quaternion algebra).  Their construction involved \textit{successive minima} (the smallest, second smallest, etc. positive integers that  are \textit{primitively} represented) of the quadratic form corresponding to the reduced norm map on the maximal order. They showed that their algorithm gives the correct answer whenever it terminates, but they did not show that the algorithm indeed halts.  Although they did not show that it halts, they were able to prove that the algorithm would halt unless there exist a pair of maximal orders satisfying a peculiar relation between their norm maps.  Roughly speaking, their algorithm halts unless there are two different maximal orders for which the first one contains more optimal embeddings of $\mathcal{O}_d$ than the second one for all $d$.  For such a pair of maximal orders, Chevyrev and Galbraith said that the first order ``dominates'' the second order.  They then conjectured that no such pair exists (see Conjecture \ref{conj:main} for a precise statement and \eqref{eqn:aOT} for the definition of the relevant quantities).
\begin{conjecture}[Chevyrev--Galbraith]\label{conj:intro}
Suppose that $\mathcal{O}$ and $\mathcal{O}'$ are maximal orders in the quaternion algebra $B_p$ ramified precisely at $p$ and $\infty$.  If $\mathcal{O}'$ ``dominates'' $\mathcal{O}$ in the sense that \eqref{eqn:optdom} holds for all $n\in\N$, then $\mathcal{O}$ and $\mathcal{O}'$ are isomorphic.  
\end{conjecture}
The goal of this paper is to prove Conjecture \ref{conj:intro}, and in turn prove the halting of the algorithm of Chevyrev and Galbraith.
\begin{thm}\label{thm:main}
Conjecture \ref{conj:intro} is true.  Furthermore, the algorithm of Chevyrev and Galbraith halts.
\end{thm}

The peculiar relation mentioned above involves the theta series of maximal orders generated by their norm maps on their trace zero elements, which are in fact ternary quadratic forms. Therefore, in order to solve our problem, some properties and facts about ternary quadratic forms and their theta series are required. As reviewed in Section \ref{sec:prelim}, by the general theory of modular forms we know that the theta series are modular forms of weight $3/2$.  Conjecture \ref{conj:intro} essentially states that if the $n$th (Fourier) coefficient of the theta series associated with one maximal order is always greater than the $n$th coefficient of the theta function associated to another maximal order, then the theta functions are the same.  Our strategy to attack the problem is to take the difference of the corresponding theta series.  Using the \textit{mass formula}, which was introduced by Siegel \cite{Siegel} and later was extended by Schulze-Pillot \cite{Schulze-Pillot}, one can show that the difference of these theta series is a cusp form and that this cusp form is orthogonal to certain functions known as unary theta functions (see Lemma \ref{lem:OTdiff}).  The central idea is to use the fact that coefficients of such forms must either vanish identically or change sign infinitely often.  These sign changes were investigated by Bruinier and Kohnen \cite{B-K} and later by Kohnen, Lau and Wu \cite{K-L-W}.

The paper is organized as follows.  In Section \ref{sec:prelim} we introduce some of the necessary background and notation for quaternion algebras and modular forms, in Section \ref{sec:defs} we give a precise statement of Chevyrev and Galbraith's conjecture, and in Section \ref{sec:proofs} we prove the their conjecture.

\section{Preliminaries}\label{sec:prelim}

In this section, we introduce some notation and give the main necessary definitions.  
\subsection{Quaternion algebras}\label{sec:quaternion}
A \begin{it}quaternion algebra\end{it} $B$ over $\Q$ is a non-commutative rank 4 algebra with the following properties (see \cite[Chapter 1]{Vigneras} for further information).  
\noindent

\noindent
\begin{enumerate}[leftmargin=*]
\item As a vector space over $\Q$, there are four generators, $1$, $\alpha$, $\beta$, and $\alpha\beta$.  
\item There exist $r,s\in\Q$ such that $\alpha^2=r$ and $\beta^2=s$.  
\item We have $\alpha\beta=-\beta\alpha$.  
\item There is an involution, known as the \begin{it}standard involution\end{it} defined for $a,b,c,d\in\Q$ by 
\[
\overline{ a +b\alpha + c\beta+d\alpha\beta} = a-b\alpha-c\beta-d\alpha\beta.
\]
\end{enumerate}
The \begin{it}reduced trace\end{it} of an element $h:=a+b\alpha+c\beta+d\alpha\beta\in B$ is 
\[
\Tr(h):=h+\overline{h} = 2a.
\]
The trace zero elements we denote by 
\[
B^0:=\left\{ h\in B: \Tr(h)=0\right\}.
\]
The \begin{it}reduced norm\end{it} of $h$ is 
\[
\Nr(h):=h\overline{h} = a^2-r b^2 -sc^2+rs d^2.
\]
The norm $\Nr$ is a quadratic form (i.e., a homogeneous degree 2 polynomial) in $4$ variables over $\Q$.  We call the quaternion algebra \begin{it}definite\end{it} if the norm map is positive-definite.  If $B$ is definite, then it is also a division algebra.  For $h\in B\setminus\Q$, the \begin{it} reduced characteristic polynomial\end{it} for $h$ is
\[
x^2-\Tr(h) x +\Nr(h).  
\]
This is the minimal polynomial of $h$ over $\Q$.  If the coefficients are furthermore in $\Z$, then we call $h$ an \begin{it}integral\end{it} element.  

An \begin{it}order\end{it} of $B$ is a rank 4 lattice (over $\Z$) of $B$ which is also a subring of $B$.  An order is called \begin{it}maximal\end{it} if it is not a proper suborder of another order of $B$.  Unlike orders in the ring of integers of a quadratic field, there may be more than one maximal order; for example, given a maximal order $\mathcal{O}$ and $h\in B$, since $B$ is non-commutative one may obtain a distinct order by conjugation.  If two maximal orders $\mathcal{O}$ and $\mathcal{O}'$ are conjugate (i.e., there exists $c\in B_p$ for which $\mathcal{O}'=c^{-1}\mathcal{O}c$ or equivalently the orders are isomorphic), then one says that they have the same \begin{it}type\end{it} and write $\mathcal{O}\sim\mathcal{O}^{\prime}$.  Note further that the elements $h$ of an order $\mathcal{O}$ are necessarily integral because for $h\in\mathcal{O}$, the sublattice $\Z[x]$ is a submodule.  

Taking the tensor product $B\otimes_{\Q}K$ with a local field $K=\R$ or $K=\Q_p$, one obtains either the ring of $2\times 2$ matrices $M(2,K)$ or a definite quaternion algebra.  The definite quaternion algebra over $K$ is unique up to isomorphism (cf. \cite[p. 31]{Vigneras}).  We say that $B$ is \begin{it}ramified\end{it} at a prime $p$ (resp. ramified at $\infty$) if $B\otimes_{\Q} \Q_p$ (resp. $B\otimes_{\Q}\R$) is definite and we say that $B$ is \begin{it}split\end{it} (or unramified) at $p$ (resp. $\infty$) otherwise.  In this paper, we are particularly interested in the quaternion algeba $B_p$ ramified precisely at $p$ and $i\infty$.  As noted above, the reduced norm on $B_p$ is a \begin{it}quaternary\end{it} ($4$-variable) quadratic form.  For a maximal order $\mathcal{O}$, the reduced norm restricted to $\mathcal{O}\cap B_p^0$ is an integral \begin{it}ternary\end{it} ($3$-variable) quadratic form.  Slightly modifying this, we define the so-called ``Gross lattice'' \cite[(12.8)]{Gross} to be 
\[
\mathcal{O}^T:=\left(2\mathcal{O}+\Z\right)\bigcap B_p^0=\left\{ 2x-\Tr(x): x\in \mathcal{O}\right\}.
\]
By \cite[Proposition 12.9]{Gross}, elements of $\mathcal{O}^T$ with norm $d$ are essentially in one-to-one correspondence with embeddings of the quadratic order $\mathcal{O}_d$ into $\mathcal{O}$.  More precisely, denote the generators of $\mathcal{O}^T$ over $\Z$ by $u_1,u_2,u_3$ and let
\begin{align}\label{eqn:aOT}
a_{\mathcal{O}^{T}}(d)&:=\#\left\{h=h_1 u_1+ h_2u_2 + h_3u_3\in \mathcal{O}^T:\Nr(x)=d,\ \mathfrak{g}(h) =1\right\}, \text{ with}\\
\label{eqn:gcd}\mathfrak{g}(h)&:=\gcd\!\left(h_{1},h_{2},h_{3}\right),
\end{align}
be the number of primitive representations of $d$ for the reduced norm $\mathrm{Nr}$ on $\mathcal{O}^{T}$.  Then 
\begin{equation}\label{eqn:primreps}
a_{\mathcal{O}^{T}}(d)= \frac{h_{\mathcal{O}}(d)}{u(d)},
\end{equation}
where $h_{\mathcal{O}}(d)$ denotes the number of optimal embeddings of $\mathcal{O}_d$ into $\mathcal{O}$ and $u(d)$ denotes the number of units in $\mathcal{O}_d$.

\subsection{Quadratic forms and theta functions}

As noted above, a quadratic form $Q$ is a homogeneous polynomial in $n$ variables of degree 2.  We may associate $Q$ with its (symmetric) Gram matrix $A$, in which case the quadratic form for $X\in Q^n$ may be written
\[
Q(X)=\frac{1}{2} X^{T} A X.
\]
We say that $Q$ is \begin{it}integral\end{it} if all of the entries of $A$ are in $\Z$ and we call $Q$ \begin{it}integer-valued\end{it} if $Q(X)\in\Z$ for all $X\in\Z^n$; to see the difference, consider $Q(X,Y)=X^2+XY+Y^2$.  We call $Q$ \begin{it}positive-definite\end{it} (resp. negative-definite) if $Q(X)\geq 0$ (resp. $Q(X)\leq 0$) for all $X\in \Q^n$ and $Q(X)=0$ if and only if $m=0$.  In this paper, we are mostly interested in positive-definite integral ternary quadratic forms.  For further information about ternary quadratic forms, a good survey may be found in \cite{Hanke}.

We split the quadratic forms into \begin{it}classes\end{it}, sets of quadratic forms which are equivalent under the action of $\GL_3(\Z)$.  Two forms $Q$ and $\mathcal{Q}$ in the same class are referred to as \begin{it}globally-equivalent\end{it} and we simply write $Q\sim_Z \mathcal{Q}$ for this relation.  Classes are then grouped together based on their local conditions.  For a positive-definite integral quadratic form ($a_{ij}\in\Z$)
\[
Q(X) = \sum_{1\leq i\leq j\leq n } a_{ij} X_i X_j,
\]
since $\Z$ embeds into the $\ell$-adic integers $\Z_{\ell}$, it is natural to allow $X\in\Z_{\ell}$ and consider $Q$ as a quadratic form over $\Z_{\ell}$ (equivalently, we may tensor the Gram matrix with $\Z_{\ell}$ over $\Z$).  Considering $Q$ over all $\Z_{\ell}$ simultaneously leads to an adelic interpretation; we do not investigate this further here, but simply note that we obtain a quadratic form $Q_{\ell}$ for each prime $\ell$.  Two quadratic forms $Q$ and $\mathcal{Q}$ are \begin{it}locally-equivalent\end{it} at the prime $\ell$ if they are equal under the action of an element of $\GL_3(\Z_{\ell})$, and we denote this equivalence by $Q\sim_{\Z_{\ell}} \mathcal{Q}$.  The set of equivalence classes which are locally-equivalent at all primes we call the \begin{it}genus\end{it} of $Q$, and (a set of representatives for) the classes in the genus we denote by $\gen(Q)$.  For the ternary case, the genus is then further subdivided into sub-genera called \begin{it}spinor genera\end{it} formed by equivalence under the spin group; see \cite[Section 102, pp. 297--305]{OMeara} for a description of this equivalence.  We use $\spn(Q)$ to denote (a set of representatives for) the classes of the spinor genus of $Q$.

For a positive-definite integral $n$-ary quadratic form $Q$ and $m\in\N_0$, let $r_Q(m)$ denote the number of representations of $m$ by $Q$.  Denoting $q:=e^{2\pi i z}$, the \begin{it}theta series\end{it} associated with $Q$ is
\begin{equation}\label{eqn:ThetaQdef}
\Theta_Q(z):=\sum_{m\in\N_0} r_Q(m)q^m= \sum_{X\in \Z^n}q^{Q(X)}.
\end{equation}
Denoting the number of \begin{it}automorphs\end{it} of $Q$ (i.e., the size of the stabilizer of $Q$ in $\GL_3(\Z)$) by $\omega_Q$, we can also define theta series
\[
\Theta_{\gen(Q)}(z):=\frac{1}{\sum_{\mathcal{Q}\in \gen(Q)}\omega_{\mathcal{Q}}^{-1}}\sum_{\mathcal{Q}\in \gen(Q)} \frac{\Theta_\mathcal{Q}}{\omega_{\mathcal{Q}}}
\]
for the genus of $Q$ and 
\[
\Theta_{\spn(Q)}(z):=\frac{1}{\sum_{\mathcal{Q}\in \spn(Q)}\omega_{\mathcal{Q}}^{-1}}\sum_{\mathcal{Q}\in \spn(Q)} \frac{\Theta_\mathcal{Q}}{\omega_{\mathcal{Q}}}
\]
for the spinor genus of $Q$.  

The theta series $\Theta_Q$ are part of a more general family of theta series, where we may insert a polynomial $P(X)$ in front of $q^{Q(X)}$.  We only need these more general theta series in the case that $n=1$, in which case for a odd character $\psi:\Z/N\Z\to \C$ and $t\in\N$ we define the \begin{it}unary theta function\end{it}
\begin{equation}\label{eqn:unarydef}
h_{\psi,t}(z):=\sum_{m\geq 1}\psi(m)mq^{tm^2}.
\end{equation}

\subsection{Modular forms}
In this paper, we view the theta series associated with quadratic forms from the perspective of (classical holomorphic) modular forms, which we require a few preliminaries to define.  
\subsubsection{Basic definitions}
Let $\H$ denote the \begin{it}upper half-plane\end{it}, i.e., those $z=x+iy\in \C$ with $x\in\R$ and $y>0$.  The matrices $\gamma=\left(\begin{smallmatrix} a&b\\ c&d\end{smallmatrix}\right)\in\SL_2(\Z)$ (the space of two-by-two integral matrices with determinant $1$) act on $\H$ via \begin{it}fractional linear transformations\end{it} $\gamma z:=\frac{az+b}{cz+d}$.  For 
\[
j(\gamma,z):=cz+d,
\]
a \begin{it}multiplier system\end{it} for a subgroup $\Gamma\subseteq \SL_2(\Z)$ and \begin{it}weight\end{it} $r\in \R$ is a function $\nu:\Gamma\mapsto \C$ such that for all $\gamma,M\in\Gamma$ (cf. \cite[(2a.4)]{Pe1})
\[
\nu(M \gamma) j(M\gamma,z)^r = \nu(M)j(M,\gamma z)^r \nu(\gamma)j(\gamma,z)^r.
\]
The \begin{it}slash operator\end{it} $|_{r,\nu}$ of weight $r$ and multiplier system $\nu$ is then 
\[
f|_{r,\nu}\gamma (z):=\nu(\gamma)^{-1} j(\gamma,z)^{-r} f(\gamma z).
\]
A \begin{it}(holomorphic) modular form\end{it} of weight $r\in\R$ and multiplier system $\nu$ for $\Gamma$  is a function $f:\H\to\C$ satisfying the following criteria:
\noindent

\noindent
\begin{enumerate}[leftmargin=*]
\item
The function $f$ is holomorphic on $\H$.
\item
For every $\gamma\in\Gamma$, we have 
\begin{equation}\label{eqn:modularity}
f|_{r,\nu}\gamma= f.
\end{equation}
\item
The function $f$ is bounded towards every \begin{it}cusp\end{it} (i.e., those elements of $\Gamma\backslash(\Q\cup\{i\infty\})$).  This means that at each cusp $\varrho$ of $\Gamma\backslash \H$, the function $f_{\varrho}(z):=f|_{r,\nu}\gamma_{\varrho}(z)$ is bounded as $y\to \infty$, where $\gamma_{\varrho}\in \SL_2(\Z)$ sends $i\infty$ to $\varrho$.  
\end{enumerate}
Furthermore, if $f$ vanishes at every cusp (i.e., the limit $\lim_{z\to i\infty} f_{\varrho}(z)=0$), then we call $f$ a \begin{it}cusp form\end{it}.  

\subsubsection{Half-integral weight forms}
We are particularly interested in the case where $r=k+1/2$ with $k\in\N_0$ and 
\[
\Gamma=\Gamma_0(M):=\left\{ \left(\begin{matrix}a&b\\ c&d\end{matrix}\right): M\mid c\right\}
\]
for some $M\in\N$ divisible by $4$.  The multiplier system is given such that there exists a character (also commonly called \begin{it}Nebentypus\end{it}) $\chi:\Z/M\Z\to \C$ for which 
\[
\frac{f(\gamma z)}{f(z)} = \chi(d) \frac{\Theta^{2k+1}(\gamma z)}{\Theta^{2k+1}(z)}.
\] 
The space of such modular forms we call the space of weight $k+1/2$ modular forms of level $4N$ and character  $\chi$ and denote the space by $M_{k+1/2}(4N,\chi)$. The subspace of cusp forms we denote by $S_{k+1/2}(4N,\chi)$.  Whenever the character is trivial, we omit it from the notation.  By \eqref{eqn:modularity} with $\gamma=T:=\left(\begin{smallmatrix} 1&1\\ 0 &1\end{smallmatrix}\right)$, we see that for $f\in M_{k+1/2}(4N,\chi)$, we have $f(z+1)=f(z)$, and hence $f$ has a Fourier expansion ($a_{f}(n)\in\C$)
\begin{equation}\label{eqn:fexp}
f(z)=\sum_{n\geq 0} a_{f}(n) e^{2\pi i n z}.
\end{equation}
The restriction $n\geq 0$ follows from the fact that $f$ is bounded as $z\to i\infty$.  One commonly sets $q:=e^{2\pi i z}$ and associates the above expansion with the corresponding formal power series, using them interchangeably unless explicit analytic properties of the function $f$ are required.

\subsubsection{Kohnen's plus space and natural operators}
We say that $f\in M_{k+1/2}(4N,\chi)$ is in \begin{it}Kohnen's plus space\end{it} \cite{Kohnenplusspace} if $a_f(n)=0$ for all $n\in\N_0$ with $(-1)^k n\equiv 2,3\pmod{4}$.  The subspace of forms in Kohnen's plus space is written $M_{k+1/2}^+(N,\chi)$ and the subspace of cusp forms is denoted by $S_{k+1/2}^+(N,\chi)$.  For every $\ell\nmid N$, there is a natural family of Hecke operators $T_{\ell^2}$, whose action on the Fourier expansion \eqref{eqn:fexp} of $S\in M_{k+1/2}^+(N,\chi)$ is given by
\[
f|T_{\ell^2}(z) := \sum_{n\geq 1}\left(a_f\!\left(\ell^2 n\right) + \chi(\ell)\left(\frac{(-1)^kn}{\ell}\right) \ell^{k-1}a_f(n) +p^{2k-1}a_f\!\left(\frac{n}{p^2}\right)\right)q^n.
\]
The operators $T_{\ell^2}$ preserve the space $S_{k+1/2}^+(N,\chi)$.  We also make use of the operator $U_{\ell^{2}}$ given by 
\[
f\big| U_{\ell^{2}}(z):=\sum_{n=1}^{\infty}a_f\!\left(n\ell^{2}\right)q^{n}.
\]
It is well-known (cf. Section 3.2 in \cite{Onotheweb}) that if $f\in S_{k+1/2}(4N,\chi)$, then  
\begin{equation}\label{eqn:Uop}
f\big| U_{\ell^{2}}\in S_{k+\frac{1}{2}}\left(4N\ell^{2},\left(\frac{4\ell^{2}}{\cdot}\right)\chi\right).
\end{equation}
Moreover, for $\ell_1,\ell_2$ relatively prime with $\ell_1\nmid N$, $T_{\ell_1^2}$ and $U_{\ell_2}^2$ commute.  Thus if $f$ is a Hecke eigenform, then $f|U_{\ell^2}$ is also a Hecke eigenform with the same eigenvalues.

\subsubsection{Theta series and modular forms}
Siegel \cite{Siegel} (see also \cite[Proposition 2.1]{Shimura}) proved that if $Q$ is an $(2k+1)$-ary quadratic form with Gram matrix $A$, then $\Theta_Q\in M_{k+1/2}(N,\chi)$ for $N\in \N$ such that $NA^{-1}$ has integral coefficients and moreover
\begin{equation}\label{eqn:thetadiff}
\Theta_Q-\Theta_{\gen(Q)}\in S_{k+1/2}(N,\chi).
\end{equation}

Moreover, by \cite[Theorem 1.44 and Proposition 3.7 (1)]{Onotheweb} or \cite[Proposition 2.1]{Shimura}, the unary theta functions $h_{t,\psi}$ defined in \eqref{eqn:unarydef} are elements of $S_{3/2}(4tN_{\psi}^2,\chi)$ for $\chi=\psi\chi_{-4}\left(\frac{4t}{\cdot}\right)$ and where $N_{\psi}$ denotes the conductor of $\psi$.  The subspace of $S_{3/2}(N,\chi)$ spanned by unary theta functions we denote by $U_{3/2}(N,\chi)$ and its orthogonal complement in $S_{3/2}(N,\chi)$ we denote by $U_{3/2}^{\bot}(N,\chi)$, where orthogonality is taken with respect to the Petersson inner product
\[
\left<f,g\right>:=\frac{1}{\left[\SL_2(\Z):\Gamma_0(4N)\right]}\int_{\Gamma_0(4N)\backslash \H} f(z)\overline{g(z)} y^{3/2} \frac{dx dy}{y^2}.
\]
Here $\left[\SL_2(\Z):\Gamma_0(4N)\right]$ is the index of $\Gamma_0(4N)$ in $\SL_2(\Z)$.   We use the fact that orthogonality from unary theta functions is preserved by $U_{\ell^2}$; this is well-known to the experts but we provide a proof for the convenience of the reader.
\begin{lem}\label{lem:unaryU}
If $f\in U_{3/2}^{\perp}(N,\chi)$ for some $N\in\N$ and character $\chi$, then \[
f\big| U_{\ell^{2}}\in U_{\frac{3}{2}}^{\perp}\left(4N\ell^{2},\left(\frac{4\ell^{2}}{\cdot}\right)\chi\right).
\]
\end{lem}
\begin{proof}
By \eqref{eqn:Uop}, $f| U_{\ell^{2}}$ is a cusp form of weight $3/2$, level $4N\ell^2$, and character $\chi':=\left(\frac{4\ell^{2}}{\cdot}\right)\chi$.  It remains to show that the projection of $f|U_{\ell^2}$ to the subspace of unary theta functions is trivial.  The basic argument is to show that if this projection is non-zero, then the coefficients of $f|U_{\ell^2}$ grow too fast. 

We first decompose
\begin{equation}\label{eqn:fUelldecomp}
f| U_{\ell^{2}}=f_0 + f_1
\end{equation}
with $f_0\in U_{3/2}(4N\ell^{2},\chi')$ and $f_1\in U_{3/2}^{\perp}(4N\ell^{2},\chi')$.  However, for $f_1\in U_{3/2}^{\perp}(4N\ell^{2},\chi')$, Duke \cite{Duke} has shown that for every $\epsilon>0$, we have 
\begin{equation}\label{eqn:Dukebound}
\left|a_{f_1}(n)\right|\ll_{f_1,\epsilon}n^{\frac{13}{28}+\epsilon}.
\end{equation}
  Suppose for contradiction that $a_{f_0}(n_0)\neq 0$ for some $n_0\in\N$.  Since 
\[
f_0=\sum_{\psi,t} \alpha_{\psi,t} h_{\psi,t},
\]
where the sum runs over $\psi$ and $t$ for which $h_{\psi,t}$ belongs to $S_{3/}(4N\ell^2,\chi')$ (in particular, the conductor of $\psi$ is a divisor of $4N\ell^2$ and $t\mid 4N\ell^2$).  By \eqref{eqn:unarydef}, we conclude that $n_0=t_0m_0^2$ for some $t_0,m_0\in\N$ with $t_0$ squarefree and
\begin{equation}\label{eqn:t0m0^2}
a_{f_0}\!\left(n_0\right) = \sum_{\psi}\sum_{\substack{\frac{t}{t_0}\in\Z^2\\ \frac{t}{t_0}\mid m_0^2}} \alpha_{\psi,t} a_{h_{\psi,t}}\!\left(t_0m_0^2\right)= \sum_{\psi}\sum_{\substack{\frac{t}{t_0}\in\Z^2\\ \frac{t}{t_0}\mid m_0^2}} \alpha_{\psi,t}\psi\left(m_0\sqrt{\frac{t_0}{t}}\right) m_0\sqrt{\frac{t_0}{t}}.
\end{equation}
Note that for any $m\equiv 1\pmod{4N\ell^2}$, we have $t/t_0\mid m_0^2m^2$ if and only if $t/t_0\mid m_0^2$ (because $t\mid 4N\ell^2$) and 
\[
\psi\left(m_0m\sqrt{\frac{t_0}{t}}\right)=\psi\left(m_0\sqrt{\frac{t_0}{t}}\right).
\]
Hence \eqref{eqn:unarydef} and \eqref{eqn:t0m0^2} imply that for any $m\equiv 1\pmod{4N\ell^2}$, we have
\begin{align*}
a_{f_0}\!\left(n_0m^2\right)&= \sum_{\psi}\sum_{\substack{\frac{t}{t_0}\in\Z^2\\ \frac{t}{t_0}\mid m_0^2m^2}} \alpha_{\psi,t}\psi\left(m_0m\sqrt{\frac{t_0}{t}}\right) m_0m\sqrt{\frac{t_0}{t}}\\
&= m\sum_{\psi}\sum_{\substack{\frac{t}{t_0}\in\Z^2\\ \frac{t}{t_0}\mid m_0^2}} \alpha_{\psi,t}\psi\left(m_0\sqrt{\frac{t_0}{t}}\right) m_0\sqrt{\frac{t_0}{t}}=ma_{f_0}\!\left(n_0\right).  
\end{align*}
Combining this with \eqref{eqn:fUelldecomp} and \eqref{eqn:Dukebound}, for $m\equiv 1\pmod{4N\ell^2}$, we obtain 
\[
a_{f}\!\left(n_0m^2\ell^2\right)=a_{f|U_{\ell}^2}\!\left(n_0m^2\right) = ma_{f_0}\!\left(n_0\right) + O\!\left(m^{\frac{13}{14}+\varepsilon}\right).
\]
Since $f\in U_{3/2}^{\perp}(4N,\chi)$, for $m$ sufficiently large this contradicts Duke's bound \eqref{eqn:Dukebound}.  This contradiction implies that $a_{f_0}(n_0)=0$ for all $n_0$, so that $f_0=0$, yielding the claim.  

\end{proof}

\section{Precise statement of Conjecture \ref{conj:intro}}\label{sec:defs}

Let $B_{p}$ over $\mathbb{Q}$ be the unique quaternion algebra which ramifies at exactly the primes $p$ and $\infty$ and let $\mathcal{O}$ be one of its maximal orders.  The algorithm by Chevyrev and Galbraith \cite{C-G} constructs an elliptic curve $E$ over $\mathbb{F}_{p^{2}}$, such that the endomorphism ring is isomorphic to the maximal order, i.e. $\mathrm{End}(E)\cong\mathcal{O}$.  
They proved that their algorithm halts unless there exists another non-conjugate maximal order $\mathcal{O}'$ for which 
\begin{equation}\label{eqn:optdom}
a_{\mathcal{O}'^T}(n)\geq a_{\mathcal{O}^T}(n)
\end{equation}
for every $n\in\N_0$, where $a_{\mathcal{O}^T}(n)$ is defined in \eqref{eqn:aOT}.  Following \cite{C-G}, we thus say that $\mathcal{O}^{\prime T}$ \begin{it}optimally dominates\end{it} $\mathcal{O}^{T}$ if \eqref{eqn:optdom} holds for all $n\in\N_0$.  Chevyrev and Galbraith then conjectured in \cite[Conjecture 1]{C-G} that no maximal order may optimally dominate another.
\begin{conjecture}[Chevyrev--Galbraith \cite{C-G}]\label{conj:main}
Let $\mathcal{O}$ and $\mathcal{O}^{\prime}$ be maximal orders of $B_{p}$. If $\mathcal{O}^{\prime T}$ optimally dominates $\mathcal{O}^{T}$, then $\mathcal{O}$ and $\mathcal{O}^{\prime}$ are of the same type.
\end{conjecture}
\begin{remarks}
\noindent

\noindent
\begin{enumerate}[leftmargin=*]
\item
Conjecture \ref{conj:main} is equivalent to Conjecture \ref{conj:intro} because all isomorphisms of orders come from conjugation.  
\item 
Paralleling the definition of type for maximal orders, we say that $\mathcal{O}^{\prime T}$ and $\mathcal{O}^{T}$ have the same type if there is a non-zero element $c\in B_{p}$ such that $c\mathcal{O}^{T}c^{-1}=\mathcal{O}^{\prime T}$, and we write $\mathcal{O}^{T}\sim\mathcal{O}^{\prime T}$. By Lemma 4 in \cite{C-G}, we know that $\mathcal{O}^{T}\sim\mathcal{O}^{\prime T}$ if and only if $\mathcal{O}\sim\mathcal{O}^{\prime}$.
\item
There is a second conjecture of Chevyrev and Galbraith about the occurrence of the smallest $n_0$ for which both $a_{\mathcal{O^{\prime}}^{T}}(n_1)\geq a_{\mathcal{O}^{T}}(n_1)$ and $a_{\mathcal{O^{\prime}}^{T}}(n_2)< a_{\mathcal{O}^{T}}(n_2)$ occur for some $n_1,n_2<n_0$.  They conjecture in particular that $n_0=O(p)$ and determine the running time of their algorithm under this assumption.  In our context, this $n_0$ corresponds to the first sign change.  Although there is some discussion in \cite{K-L-W} about the size of $n_0$, there are a number of inexplicit constants which would need to be worked out to determine the size of $n_0$ implied by their theorem, and it is not expected that their proof would yield a bound anywhere close to the conjectured $O(p)$.  The first author is trying to determine (and improve upon) an explicit bound for $n_0$ in his Masters thesis.  
\item
\label{a and r rlts}By Lemma 11 in \cite{C-G}, we have $a_{\mathcal{O^{\prime}}^{T}}(n)\geq a_{\mathcal{O}^{T}}(n)$
for all $n$ if and only if $r_{\mathcal{O^{\prime}}^{T}}(n)\geq r_{\mathcal{O}^{T}}(n)$
for all $n$.
\end{enumerate}
\end{remarks}

\section{Proof of Theorem \ref{thm:main}}\label{sec:proofs}

Recall that for a maximal order $\mathcal{O}$ of $B_p$, the associated reduced norm $\Nr$ on $\mathcal{O}^T$ is a positive-definite integral ternary quadratic form $Q_{\mathcal{O}^T}$.  Gross \cite[(12.8)]{Gross} constructed the associated theta series 
\begin{equation}\label{eqn:thetaOT}
\vartheta_{\mathcal{O}^T}:=\Theta_{Q_{\mathcal{O}^T}},
\end{equation}
which is an element of Kohnen's plus space $M_{3/2}^+(p)$.  The following lemma plays a key role in the proof of Conjecture \ref{conj:main}.
\begin{lem}\label{lem:OTdiff}
If $\mathcal{O}$ and $\mathcal{O}'$ are two maximal orders in the quaternion algebra $B_p$, then 
\[
\vartheta_{\mathcal{O}^T}-\vartheta_{\mathcal{O}'^T}\in S_{\frac{3}{2}}^+(p).
\]
Furthermore, $\vartheta_{\mathcal{O}^T}-\vartheta_{\mathcal{O}'^T}\in U_{3/2}^{\bot}(4p)$.  

\end{lem}
\begin{proof}
As noted by Gross (see \cite[p. 130]{Gross}), the maximal orders of $B_p$ are all locally conjugate over $\Z_{\ell}$, from which we conclude that for all primes $\ell$
\[
Q_{\mathcal{O}^T}\sim_{\Z_{\ell}}Q_{\mathcal{O}'^T}.
\]
Thus $Q_{\mathcal{O}^T}$ and $Q_{\mathcal{O}'^T}$ are in the same genus by definition.  Hence, by \eqref{eqn:thetadiff},
\begin{multline*}
\vartheta_{\mathcal{O}^T}-\vartheta_{\mathcal{O}'^T} = \Theta_{Q_{\mathcal{O}^T}}-\Theta_{\gen\!\left(Q_{\mathcal{O}^T}\right)} + \Theta_{\gen\!\left(Q_{\mathcal{O}^T}\right)}-\Theta_{Q_{\mathcal{O}'^T}}\\
 =  \Theta_{Q_{\mathcal{O}^T}}-\Theta_{\gen\!\left(Q_{\mathcal{O}^T}\right)} + \Theta_{\gen\!\left(Q_{\mathcal{O}'^T}\right)}-\Theta_{Q_{\mathcal{O}'^T}} =  \left(\Theta_{Q_{\mathcal{O}^T}}-\Theta_{\gen\!\left(Q_{\mathcal{O}^T}\right)}\right) + \left(\Theta_{\gen\!\left(Q_{\mathcal{O}'^T}\right)}-\Theta_{Q_{\mathcal{O}'^T}}\right)
\end{multline*}
is a cusp form.  Moreover, it is contained in Kohnen's plus space of level $p$ by construction.  

It remains to show that $\vartheta_{\mathcal{O}^T}-\vartheta_{\mathcal{O}'^T}$ is orthogonal to unary theta functions.  However, since $p$ is squarefree and odd, Kohnen has proven in \cite[Theorem 2]{Kohnenplusspace} that $S_{3/2}^+(p)$ is Hecke-isomorphic to $S_{2}(p)$ under a linear combination of the Shimura lifts defined in \cite{Shimura} (and hence has a basis of simultaneous Hecke eigenforms).  Since any element of $S_{3/2}^+(p)$ may be written as a linear combination of Hecke eigenforms, it suffices to show that all of the Hecke eigenforms are orthogonal to unary theta functions.  

Next recall that the Hecke operators are Hermitian with respect to the Petersson inner product (see \cite[Section 3]{Kohnenplusspace}).  Denoting the eigenvalue of $h_{t,\psi}$ under the Hecke operator $T_{\ell^2}$ by $\lambda_{\ell}$ and the eigenvalue of an eigenform $f$ in $S_{3/2}^+(p)$ by $\lambda_{f,\ell}$, we see that
\begin{equation}\label{eqn:innerortho}
 \lambda_{\ell}\left<h_{t,\psi},f\right> = \left<h_{t,\psi}|T_{\ell^2},f\right> =\left<h_{t,\psi},f|T_{\ell^2}\right> =\overline{\lambda_{f,\ell}}\left<h_{t,\psi},f\right>.
\end{equation}
We conclude that if $h_{t,\psi}$ and $f$ are not orthogonal, then $\lambda_{\ell}=\lambda_{f,\ell}$ for all $\ell$, where we use the fact that the eigenvalues must be real because the Hecke operator is Hermitian.  However, the elements of $U_{3/2}(4p)\subset S_{3/2}(4p)$ have the same eigenvalues as weight $2$ Eisenstein series and $\lambda_{f,\ell}$ is the eigenvalue for a weight $2$ cusp form by Kohnen's Hecke-isomorphism.  The eigenvalues cannot always coincide and therefore $h_{t,\psi}$ and $f$ are orthogonal.

\end{proof}

The strategy of our proof is to study the sign changes of the Fourier coefficients of the differences $\vartheta_{\mathcal{O}^{\prime T}}-\vartheta_{\mathcal{O}^{T}}$.  For this, we require \cite[Theorem 1]{K-L-W} of Kohnen, Lau, and Wu.
\begin{thm}[Kohnen, Lau and Wu]\label{thm: Kohnen-Lau-Wu}
Let $N\geq4$ an integer divisible by $4$ and $\chi$ be a Dirichlet character modulo $N$. If $g\in U_{3/2}^{\bot}(N,\chi)$, then for any positive squarefree integer $t$ such that $a_g(t)\neq0$ and the sequence $\left\{ a_g(tn^{2})\right\} _{n\in\mathbb{N}}$ is real, the sequence $\left\{ a_g(tn^{2})\right\} _{n\in\mathbb{N}}$ contains infinitely many sign changes.
\end{thm}
\begin{remarks}\label{rem:signchange}
\noindent

\noindent
\begin{enumerate}[leftmargin=*]
\item
Kohnen, Lau and Wu actually gave much stronger results in their paper \cite{K-L-W} but this simplified version is strong enough for our use.
\item 
One can use an argument involving the sign changes to directly show that $\vartheta_{\mathcal{O}'^T}-\vartheta_{\mathcal{O}^T}\in U_{3/2}^{\bot}(4p)$ if $\mathcal{O'}^T$ optimally dominates $\mathcal{O}^T$.  To illustrate the usage of Theorem \ref{thm: Kohnen-Lau-Wu}, we briefly sketch the proof; further details may be found in the first author's upcoming Masters thesis.  One sees directly from \eqref{eqn:unarydef} that the coefficients of unary theta functions alternate in sign.  Using a bound of Duke \cite{Duke} for the coefficients of elements of $U_{3/2}^{\bot}(4p)$, the coefficients of the difference $\vartheta_{\mathcal{O}^T}-\vartheta_{\mathcal{O}'^T}$ are dominated by the coefficients of the contribution from unary theta functions and hence alternate unless the contribution from $U_{3/2}^{\bot}(4p)$ is trivial.  However, slightly abusing notation by abbreviating
\[
r_{\mathcal{O}^T}(n):=r_{Q_{\mathcal{\mathcal{O}}^T}}(n),
\]
we may split the elements of $h\in \mathcal{O}^T$ by $\mathfrak{g}(h)=f$ (see \eqref{eqn:gcd}) to obtain 
\begin{equation}\label{eqn:rOTeval}
r_{\mathcal{O}^T}(n) = \sum_{\substack{f\in\Z\\ f^2\mid n}} a_{\mathcal{O}^T}\!\left(\frac{n}{f^2}\right).
\end{equation}
Hence if $\mathcal{O}'^T$ optimally dominates $\mathcal{O}^T$, then $r_{\mathcal{O}'^T}(n)\geq r_{\mathcal{O}^T}(n)$, and we conclude that the contribution from unary theta functions is trivial.  
\end{enumerate}
\end{remarks}

The next proposition is a key step in the proof of Theorem \ref{thm:main}.
\begin{prop}\label{prop: 1}
Let $\mathcal{O}$ and $\mathcal{O}^{\prime}$ be maximal orders of $B_{p}$. If $\mathcal{O}^{\prime T}$ optimally dominates $\mathcal{O}^{T}$, then $\vartheta_{\mathcal{O}^{\prime T}}(z)=\vartheta_{\mathcal{O}^{T}}(z)$.
\end{prop}
Write
\[
g(z):=\vartheta_{\mathcal{O}^{\prime T}}(z)-\vartheta_{\mathcal{O}^{T}}(z).
\]
By Lemma \ref{lem:OTdiff}, $g\in U_{3/2}^{\bot}(4p)$, and we have $a_g(n)\geq0$ for all $n\in\N$ by assumption. Hence to conclude Proposition \ref{prop: 1}, it suffices to prove the following slightly stronger proposition.
\begin{prop}
If $g\in U_{3/2}^{\perp}(4N,\chi)$ for some $N\in\N$ and character $\chi$ and $a_g(n)\geq 0$ for all $n$, then $g=0$.
\end{prop}
\begin{proof}
We show the claim by proving that $a_g(n)=0$ for all $n\in\N$.  To give the idea of the argument suppose that there exists a squarefree $t\in\mathbb{N}$ such that $a_g(t)\neq0$, then by Theorem \ref{thm: Kohnen-Lau-Wu}, the sequence $\left\{ a_g(tm^{2})\right\}_{m\in\mathbb{N}}$ has sign changes. But then this contradicts the fact that $a_g(n)\geq 0$ for all positive $n$. Hence we have $a_g(n)=0$ for all squarefree $n\in\N$.

We proceed similarly to show that $a_{g}(n)=0$ for $n=tm_0^2$ with $t$ squarefree and 
\[
m_0=\prod_{j=1}^{J} \ell_j\in\N,
\]
where $\ell_j$ are (not necessarily distinct) primes.  Suppose for contradiction that $a_{g}(tm_0^2)\neq 0$.  Denoting 
\[
U_{m_0^2}:=\prod_{j=1}^J U_{\ell_j^2}
\]
and repeatedly using Lemma \ref{lem:unaryU}, there exists a character $\chi'$ for which 
\[
g|U_{m_0^2}\in U_{\frac{3}{2}}^{\perp}\!\left(4pm_0^2,\chi'\right).
\]
Thus we may apply Theorem \ref{thm: Kohnen-Lau-Wu} to $g|U_{m^2}$ to conclude that $\{a_{g|U_{m_0^2}}(tm^2):m\in\Z\}$ has infinitely many sign changes. However, since 
\[
a_{g|U_{m_0^2}}\!\left(tm^2\right)= a_g\left(tm_0^2m^2\right)\geq 0,
\] 
we obtain a contradiction. Thus $a_g(tm_0^2)=0$, as desired. 
\end{proof}
We have now established most of the ingredients necessary to prove Theorem \ref{thm:main}.  The main remaining piece is an equivalence between theta series $\vartheta_{\mathcal{O}^T}$ and $\vartheta_{\mathcal{O}'^T}$ agreeing and $\mathcal{O}^T$ and $\mathcal{O}'^T$ having the same type. 
\begin{lem}
\label{lem: equi}
Let $\mathcal{O}$ and $\mathcal{O}^{\prime}$ be maximal orders of $B_{p}$.  Then the following statements are equivalent:

\begin{enumerate}[leftmargin=*]
\item[\rm (a)] $\mathcal{O}^{T}\sim\mathcal{O}^{\prime T}$;
\item[\rm (b)] $\vartheta_{\mathcal{O}^{T}}=\vartheta_{\mathcal{O}^{\prime T}}$;
\item[\rm (c)] $Q_{\mathcal{O}^T}\sim_{\mathbb{Z}}Q_{\mathcal{O}'^T}$.
\end{enumerate}
\end{lem}
\begin{proof}

(a)$\Rightarrow$(b): Suppose that there is a non-zero element $c\in B_{p}$ such that $c\mathcal{O}^{T}c^{-1}=\mathcal{O}^{\prime T}$. Since 
\[
\Nr\!\left(cXc^{-1}\right)=\Nr(X)
\]
 for all $X\in B_{p}$ and non-zero $c\in B_{p}$, we conclude (b) by the definition \eqref{eqn:thetaOT} of the theta series.

(b)$\Rightarrow$(c): If $\vartheta_{\mathcal{O}^{T}}(z)=\vartheta_{\mathcal{O}^{\prime T}}(z)$, then all the coefficients of their Fourier expansions are the same.  By Schiemann \cite{schiemann}, we have $Q_{1}\sim_{\mathbb{Z}}Q_{2}$ (actually, Schiemann gave a much stronger result; roughly speaking, it only requires the first few coefficients of the Fourier expansions to be the same).  

(c)$\Rightarrow$(a): This is shown in \cite[Section 4]{G-L} by defining the associated ternary quadratic form on \cite[p. 1473]{G-L} and then showing that the map forms a bijection between orbits under $\GL_3(\Z)$ and isomorphism classes of quaternion rings over $\Z$ in \cite[Proposition 4.1]{G-L}.

\end{proof}

We are finally ready to prove our main theorem, which we state again for the convenience of the reader.
\begin{thm}
Let $\mathcal{O}$ and $\mathcal{O}^{\prime}$ be maximal orders of $B_{p}$. If $\mathcal{O}^{\prime T}$ optimally dominates $\mathcal{O}^{T}$, then $\mathcal{O}$ and $\mathcal{O}^{\prime}$ are of the same type.  Furthermore, the algorithm of Chevyrev and Galbraith halts.
\end{thm}
\begin{proof}[Proof of Theorem \ref{thm:main}]
By Proposition \ref{prop: 1}, if $\mathcal{O}^{\prime T}$ optimally dominates $\mathcal{O}^T$, then $\vartheta_{\mathcal{O}^T}=\vartheta_{\mathcal{O}^{\prime T}}$.  Hence by the equivalence of (b) and (a) in Lemma \ref{lem: equi}, we obtain that $\mathcal{O}^T\sim \mathcal{O}^{\prime T}$.  Finally, by Lemma 4 of \cite{C-G}, we conclude that $\mathcal{O}$ and $\mathcal{O}'$ have the same type.
\end{proof}

\end{document}